\documentclass[10pt,a4paper]{article}
\usepackage{amssymb,amsfonts,amsmath,amsthm,amscd}
\usepackage{tabularx}
\usepackage{graphicx}
\usepackage{mathtools}
\usepackage{indentfirst}
\usepackage{titlefoot}

\usepackage{comment}

\DeclareMathOperator{\diag}{diag}

\def\Chn{\mathbb{C}H^n}
\def\Ch2{\mathbb{C}H^2}
\def\Cg2{{\mathcal {CH}}^2}
\def\Cg{{\mathcal {CH}}^n}
\def\Rhn{\mathbb{R}H^n}
\def\Rh4{\mathbb{R}H^4}
\def\Rg2{{\mathcal {RH}}^4}

\def\chn{ch_n}
\def\Hnn{H ^{2n-1}}
\def\hh{\mathfrak{h} _{2n-1}}

\def\RR{\mathbb{R}}
\def\CC{\mathbb{C}}

\def \tr{\text{tr}}
\def \ad{\text{ad}}
\def \Ric{\text{Ric}}
\def \Der{\text{Der}}
\def \nn {\mathfrak{n}}
\def \aa {\mathfrak{a}}
\def \ss {\mathfrak{s}}

\newtheorem{Theorem}{{Theorem}}[section]
\newtheorem{Lemma}{{Lemma}}[section]
\newtheorem{Corollary}{{Corollary}}[section]
\newtheorem{Remark}{{Remark}}[section]
\newtheorem{Definition}{{Definition}}[section]

\newtheorem{innercustomthm}{Proposition}
\newenvironment{customthm}[1]
  {\renewcommand\theinnercustomthm{#1}\innercustomthm}
  {\endinnercustomthm}

\title{Classification of Left Invariant Riemannian metrics on Complex hyperbolic space}
\author
{
Andrijana Deki\'{c}\thanks{Corresponding author: andrijana.savkovic@gmail.com },
Marijana Babi\'{c}, 
Srdjan Vukmirovi\'{c}
}

\date{}

\begin{document}

\maketitle
\unmarkedfntext{2020 \emph{Mathematics Subject Classification}.
53C30, 58D17, 22E60, 22E25.

\indent\phantom{K}\emph{Key words and phrases}: Complex hyperbolic space,  left invariant metric, solvmanifold, Ricci soliton}

\begin{abstract}
It is well known that $\Chn$ has the structure of solvable Lie group with left invariant metric of constant holomorphic sectional curvature.
In this paper we give the full classification of all possible left invariant Riemannian metrics on this Lie group.
We prove that all of these metrics are of constant negative scalar curvature and only one of them is Einstein (up to isometry and scaling). Finally, we present the relation between Ricci solitons on Heisenberg group and Einstein metric on $\Chn$.
\end{abstract}

\section* {Introduction}

Since $\Chn$ is symmetric space of negative sectional curvature, by the classical results of Heintze \cite{Heintze, Heintze2}, it may be viewed as a connected solvable real Lie group with a left invariant metric. This group, denoted by $\Cg$, is the noncompact part of the Iwasawa decomposition  $SU(1,n) = KAN$  of the isometry group of Complex hyperbolic space.
The compact part is isomorphic to $U (n)$, the nilpotent part $N$ is Heisenberg group $\Hnn$, and the abelian part is 1-dimensional. The semidirect product $\Cg=AN$ acts simply transitively on $\Chn$, giving it a structure of a Lie group with the left invariant metric inherited form $\Chn$.

Now an interesting question arises: what are all other possible left invariant Riemannian metric on this Lie group?

There are two standard approaches to the classification problem based on the moduli spaces of left invariant Riemannian metrics on a given Lie group. One is to fix Lie algebra commutators, then consider the space of all possible left invariant metrics, and find the simplest representatives under the action of the automorphism group.  This approach we use to classify all the non-isometric left invariant Riemannian metrics on $\Cg$. It relies on the fact that $\Cg$ is completely solvable, therefore, by the results of Alekseevsky \cite{Ale1,Ale2}, the isometry classes of left invariant metrics are exactly the orbits of the automorphism group acting on the space of left invariant metrics. Slightly different but equivalent approach is to fix an orthonormal (or pseudo-orthonormal) base, and classify all possible Lie brackets up to the action of the automorphism group. In the space of all bilinear skew-symmetric forms, the Jacobi identity defines the hypersurface of admittable Lie brackets. The isometry classes are again the orbits of the automorphism group. Both approaches are used by Lauret in search for distinguished left invariant metrics (Ricci solitons) on nilpotent Lie groups \cite{Lauretnilsoliton}. The first method is systematically outlined in \cite{TamaruOrbite}, and the second one in \cite{TamaruMilnorType}, where Tamaru with coauthors called it Milnor-type theorems in reference to the grounding work of Milnor \cite{Milnor}, who used it for the first time to classify all left invariant Riemannian metrics on three-dimensional unimodular Lie groups.
Although Milnor's method relies on existence of the cross product in dimension three, lots of results have been later obtained for Riemannian and Lorentzian cases in dimensions three and four (see for example \cite{BB,neda,calvaruso,ChPapa,CorParker,Lauretdim34,Rahmani}), and recently, for dimension four with neutral signature as well \cite{Tijana4sol,Tijana4nil}.

Results in arbitrary dimension are more recent. All the left invariant Riemannian and Lorentzian metrics on Heisenberg group were classified in \cite{Srdjan}. Pseudo-Riemannian metrics of Real hyperbolic space modelled as a Lie group, have been considered both by the variation of Lie brackets \cite{TamaruPseudoR}, and by the variation of inner products \cite{TijanaiSrdjan}.
It has been shown in \cite{LauretOne} that the only connected and simply-connected Lie groups admitting only one left-invariant Riemannian metric up to scaling and isometry are the Euclidean
space, Real hyperbolic space and $H_3\times \mathbb R^n$ (product of three-dimensional Heisenberg group and Euclidian space). The Lorentzian metrics on $H_3\times \mathbb R^n$ have been classified in \cite{TamaruHxR}.

Algebraic Ricci solitons on nilpotent and solvable Lie groups are introduced by Lauret \cite{Lauretnilsoliton, Lauretsolsoliton}. Since the Lie group $\Cg$ is solvable one-dimensional abelian extension of the nilpotent Heisenberg group, it is interesting to see the explicit relation of these two groups in the context of  more general Lauret's work.
Using the classification of all left invariant Riemannian and Lorentzian metrics on arbitrary-dimensional Heisenberg group from \cite{Srdjan}, Nasehi classified nilsolitons on Heisenberg group \cite{Nasehi}. In Riemannian case, the result is conveniently simple, namely, there is only one nilsoliton up to action of the automorphism group and scaling. The extension of this nilsoliton gives rise to the Einstein metric on $\Cg$. This is an example of codimension one Ricci soliton subgroups of solvable Iwasawa groups, classified recently by Dom\'inguez-V\'azquez et.\ al.\ \cite{Vazquez}.\\

This paper is orgainzed as follows: In Section \ref{sec:preliminaries} we introduce the basic notation. Metric Lie groups and algebras are defined, together with isometry classes among them.

The main result of this paper is the classification of Riemannian left invariant metrics on $\Cg$, presented in Section \ref{sec:classification}, Theorem \ref{th:metrike}. We describe the group of automorphisms of it's Lie algebra $Aut(\chn)$ in Lemma \ref{le:auto} and Corollary \ref{cor:podgrupe}.
It is shown that it contains the Symplectic group $Sp(2n-2,\RR)$, which plays important role in the proof of the classification theorem, i.e.\ it permits a partial diagonalization by the symplectic eigenvalues.

In Section \ref{sec:geometry} we explore the geometrical properties of Lie group $\Cg$ equipped with various left invariant metrics $g(p,x,\sigma,\beta)$. This provides a whole class of Riemannian solvmanifolds which might be interesting for various further research. The curvature tensor is given both explicitly and on the exterior algebra. None of the metrics is flat and all of them have a constant negative sectional curvature. There is only one Einstein metrics up to automorphisms and scaling, and we show how it can be obtained from Heisenberg nilsoliton.

\section{Preliminaries}
\label{sec:preliminaries}

Complex hyperbolic space is non-compact rank-one symmetric space of negative sectional curvature:
$$\Chn =  SU(1,n)/S(U(1)\times U(n)).$$
Therefore, it is a solvmanifold, i.e.\ it can be represented as a connected solvable Lie group with a left invariant metric \cite{Heintze, Heintze2}. This group is semidirect product of abelian and nilpotent part (Heisenberg group) of the Iwasawa decomposition of its isometry group:
$$\Cg=\RR \ltimes \Hnn.$$

Lie algebra $\chn$  of the Lie group $\Cg$ is semidirect product of abelian and Heisenberg algebra
\begin{equation}
\label{eq:semi}
\chn = \RR \ltimes \hh.
\end{equation}
It  is spanned by vectors $X, Y_1,\dots ,Y_{n-1}, Z_1,\dots, Z_{n-1}, W$ with nonzero commutators:
\begin{equation}
\label{com_chn}
\begin{gathered}
[X, Y_i]=\frac{1}{2}Y_i,\quad [X, Z_i]=\frac{1}{2}Z_i,\quad [X, W]=W, \quad [Z_j, Y_i]=\delta_{ij}W,\\
 i,j \in\{1,\dots,n-1\}.
\end{gathered}
\end{equation}
This algebra is 3-step solvable and has 1-dimensional center spanned by the vector W.

If  we use identification of  $\CC ^n \cong \RR ^{2n}$ of the form
$$(z_1, \dots, z_n) \mapsto (x_1, \dots, x_n, y_1, \dots , y_n),\enskip z_k = x_k + iy_k, \enskip k \in \{1, \dots , n\},$$
then the multiplication by $i$ on $\CC ^n $ induces the standard complex structure on $\RR ^{2n}$, given by the matrix
$$
J_n =
\begin{pmatrix}
0 & I_n\\
-I_n & 0
\end{pmatrix},
$$
where $I_n$ is $n\times n$ identity matrix.\\
The standard symplectic form in vector space $\RR ^{2n}$ is
$$\omega (u,v)  = u^TJ_nv, \quad u,v \in \RR ^{2n}.$$
Symplectic group is the group of all linear transformations of $\RR ^{2n}$ preserving $\omega$, that is
\begin{align}
\label{SymplecticGroup}
Sp (2n, \RR) &= \{ F \in Gl_{2n}(\RR)  \ | \  \omega (Fu, Fv) = \omega (u,v)  \} \nonumber \\
&= \{ F \,  | \,  F^TJ_nF =J_n \}.
\end{align}

\begin{Definition}
Two metric Lie algebras (i.e. Lie algebras equipped with inner products) are\\
a) {\em isometric} if there exists a homomorphism of vector spaces that preserves curvature tensor and its covariant derivatives.\\
b) {\em isomorphic} if they are isometric and the isometry preserves Lie algebra commutators (i.e.\ they are isomorphic as Lie algebras as well).
\end{Definition}

The relation between isometry of metric Lie algebras and the isometry of corresponding metric Lie groups (Lie group with the left invariant metric) is given by:

\begin{Lemma}[\cite{Ale2}]
\label{le:aleale}
Two metric Lie algebras  are isometric if and only if the corresponding Riemannian spaces are isometric.
\end{Lemma}

What makes it possible to classify all non-isometric inner products is the following Lemma and the fact that Lie algebra $\chn$ is completely solvable:

\begin{Lemma}[\cite{Ale1, Ale2}]
\label{le:ale}
Isometric completely solvable metric Lie algebras are isomorphic.
\end{Lemma}

\section{Classification of left invariant Riemannian metrics on $\Cg$}
\label{sec:classification}

By left translations, every left invariant inner product on a Lie algebra uniquely determines a left invariant metric on a Lie group. Therefore, the problem of  classification of all non-isometric left invariant Riemannian metrics on the Lie group $\Cg$ is equivalent to classification of all non-isometric positive definite inner products on the Lie algebra $\chn$.

In a fixed basis with commutators (\ref{com_chn}), inner products are represented by the symmetric positive definite matrices S. The set of orbits of the automorphism group $Aut(\chn)$ acting of the space of inner products is called \textit{Moduli space} in \cite{TamaruOrbite}. Every orbit is in one-to-one correspondence to a class of equivalence of the relation
 $S\sim F^TSF$,  $F\in Aut(\chn)$.
Thus, the classification of non-isometric metrics is equivalent to finding the simplest representatives of these orbits. One of the key roles in the proof of Classification theorem is the fact that $Aut(\chn)$ contains the Symplectic group $Sp(2n-2,\RR)$, which allows the diagonalization by symplectic eigenvalues, and thus significantly simplifying the set of representatives.

\begin{Lemma}
\label{le:auto}
The automorphisms group of the Lie algebra $\chn$   is
\begin{equation}
\label{eq:aut}
\begin{split}
Aut(\chn)&=
\{
\begin{pmatrix}
1 & 0 & 0  \\
u & M & 0 \\
a & v^T & \lambda
\end{pmatrix}
 \, | \,
M \in GL(2n-2, \RR), \enskip  M^T J_{n-1} M = \lambda J_{n-1}, \\
& v\in \RR ^{2(n-1)}, \enskip a, \lambda \in \RR, \enskip  \lambda \neq 0, \enskip  u=\frac{1}{2\lambda}MJ_{n-1} v
\}.
\end{split}
\end{equation}

\end{Lemma}

\begin{proof}
Fix a basis $(X, Y_1,\dots ,Y_{n-1}, Z_1,\dots, Z_{n-1}, W)$ with commutators of the form (\ref{com_chn}).
By definition, every automorphism $F\in\chn$ must preserve commutators, i.e.
$$[FX_1, FX_2]=F[X_1,X_2], \qquad X_1,X_2\in \chn.$$
It also preserves center and the derived algebra. In the case of $\chn$ center is one-dimensional, spanned by the vector W. The derived algebra is $\hh$, the Heisenberg algebra spanned by the vectors
$Y_1,\dots ,Y_{n-1}, Z_1,\dots, Z_{n-1}, W$. This directly simplifies automorphisms represented in a bloc matrix form:
$$ F=
\begin{pmatrix}
1 & 0 & 0  \\
u & M & 0 \\
a & v^T & \lambda
\end{pmatrix}.
$$
Futhermore, the relations\\
$$[FZ_j, FY_i]=\delta_{ij}FW, \quad [FY_i, FY_j]=0, \quad [FZ_i, FZ_j]=0, \quad  i,j \in\{1,\dots,n-1\}$$
impose a restriction on $M$, that is $M$ has to be almost symplectic: $$ M^T J_{n-1} M = \lambda J_{n-1}.$$
Finally, the relations:
$$[FX, FY_i]=\frac{1}{2}FY_i,\quad [FX, FZ_i]=\frac{1}{2}FZ_i,
\quad i,j \in\{1,\dots,n-1\},$$
after short calculation, give the correlation between vectors $u$ and $v$:
$$ u=\frac{1}{2\lambda}MJ_{n-1} v.$$
\end{proof}

Note that all the automorphisms are divided in three types. We call them \textbf{diagonal, symplectic} and \textbf{generalized translations}.

\begin{Corollary}
\label{cor:podgrupe}
The identity component of the automorphism group of the Lie algebra $\chn$ is
$$Aut_0(\chn)=D \ltimes ( Sp \ltimes T ),$$
where
\begin{align*}
D&=\left \{
\begin{pmatrix}
1 & 0 & 0  \\
0 & \alpha I & 0 \\
0 & 0 & \alpha^2
\end{pmatrix}  | \, \alpha>0
\right \},\\
Sp &=\left \{
\begin{pmatrix}
1 & 0 & 0  \\
0 & M & 0 \\
0 & 0 & 1
\end{pmatrix}  | \, M^TJ_{n-1}M=J_{n-1}
\right \}
\cong Sp(2n-2,\RR),\\
T&=\left \{
\begin{pmatrix}
1 & 0 & 0  \\
\frac{1}{2}J_{n-1}v &  I & 0 \\
a & v^T & 1
\end{pmatrix}  |  \, a\in \RR, v\in \RR^{2n-2}
\right \}.
\end{align*}

\end{Corollary}

Now we may proceed to the main theorem:

\begin{Theorem}[Classification theorem]
\label{th:metrike}
All non-isometric positive definite inner products on Lie algebra $\chn$,
in a basis with commutators of the form (\ref{com_chn}),
are represented by the matrices:

\begin{equation}
\label{eq:Schn}
S(p,x,\sigma,\beta)=
\begin{pmatrix}
p &x^T& 0 & 0\\
x& \sigma & 0 & 0\\
0 & 0 &\sigma&0\\
0 & 0 & 0 & \beta
\end{pmatrix} ,
\end{equation}
where $\enskip p, \beta >0, \quad x=(x_1, \dots, x_{n-1})^T \in \mathbb R^{n-1}, \quad x_i \geq 0,$\\
$\sigma=\diag(\sigma_1, \dots, \sigma_{n-2},1), \quad \sigma_1\ge \dots \ge \sigma_{n-2}\ge1$.
\end{Theorem}

\begin{proof}

We are looking for the simplest representatives of the orbits of the automorphism group  $Aut(\chn)$ acting on the space of positive definite symmetric matrices.

Using the notation from the Corollary \ref{cor:podgrupe} for the subgroups of the group $Aut(\chn)$, denote respectively:
$F_d(\alpha)\in D$, $F_{Sp}(M)\in Sp$, $F_t(v,a)\in T.$

Consider arbitrary positive definite symmetric $2n\times2n$ matrix $S$,  in a block matrix form
$$ S=
\begin{pmatrix}
p_1 & z^T & q  \\
z & \bar{S} & w \\
q & w^T & {\beta _1}
\end{pmatrix}.
 $$

A generalized translation  $F_{t_1} (-\frac{1}{\beta _1}w,0)\in T$ simplifies $S$  to
$$
S_1 = F_{t_1} ^T S F_{t_1} =
\begin{pmatrix}
p_2 & z_1^T & q_1  \\
z_1 & \bar{S}_1 & 0 \\
q_1 & 0 & \beta _1
\end{pmatrix} .
$$
Since  $S_1$ is positive definite, $\bar{S}_1$ is also a positive definite  $(n-2)\times(n-2)$ matrix. By the Williamson theorem  \cite{Williamson}, any Riemannian matrix can be diagonalized by some symplectic matrix  $M_1$. Furthermore, the diagonalizing symplectic matrix can be selected in a way that allows any ordering of symplectic eigenvalues (see Theorem 2.1 from \cite{Srdjan}). Therefore, we choose the symplectic automorphism $ F_{Sp_2}(M_1)\in Sp$, to obtain

$$
S_2=F_{Sp_2}^TS_1 F_{Sp_2}=
\begin{pmatrix}
p_2 & z_1^TM_1 & q_1  \\
M_1^Tz_1 & M_1^T\bar{S}_1M_1 & 0 \\
q_1 & 0 & \beta _1
\end{pmatrix}=
\begin{pmatrix}
p_2 & z_2^T & q_1  \\
z_2 & D_{\bar{\sigma}} & 0 \\
q_1 & 0 & \beta _1
\end{pmatrix},
$$
where $D_{\bar{\sigma}}$ is the diagonal matrix
$
D_{\bar{\sigma}}= \diag(\bar{\sigma}_1, ...,\bar{\sigma}_{n-1}, \bar{\sigma}_1,...,\bar{\sigma}_{n-1}),
$
and\\
$\bar{\sigma}_1 \geq..., \geq \bar{\sigma}_{n-1}>0$ are the  symplectic eigenvalues of matrix $\bar{S}_1.$\\
Now, we choose a simple translation $F_{t_3 }(\vec{0},-\frac{q_1}{\beta _1})\in T$ to obtain

$$S_3=F_{t_3 }^TS_2F_{t_3 }=
\begin{pmatrix}
p_2-\frac{q_1^2}{\beta_1} & z_2^T & 0  \\
z_2 & D_{\bar{\sigma}} & 0 \\
0 & 0 & \beta _1
\end{pmatrix}
=
\begin{pmatrix}
p & z_2^T & 0  \\
z_2 & D_{\bar{\sigma}} & 0 \\
0 & 0 & \beta _1
\end{pmatrix} .
$$
This is further simplified by a diagonal automorphism $F_{d_4}((\bar{\sigma}_{n-1})^{-\frac{1}{2}})\in D$

$$
S_4=F_{d_4}^TS_3F_{d_4}=
\begin{pmatrix}
p & \alpha z_2^T & 0  \\
\alpha z_2 & \alpha^2D_{\bar{\sigma}} & 0 \\
0 & 0 & \alpha^2\beta _1
\end{pmatrix}=
\begin{pmatrix}
p & \bar{z}^T & 0  \\
\bar{z} & \bar{D} & 0 \\
0 & 0 & \beta
\end{pmatrix},
$$
where
\begin{align*}
\bar{D}&=\diag (\frac{\bar{\sigma}_1}{\bar{\sigma}_{n-1}}, \dots, \frac{\bar{\sigma}_{n-2}}{\bar{\sigma}_{n-1}},1, \frac{\bar{\sigma}_1}{\bar{\sigma}_{n-1}}, \dots, \frac{\bar{\sigma}_{n-2}}{\bar{\sigma}_{n-1}},1)\\
&=
\diag(\sigma_1,\dots,\sigma_{n-2}, 1, \sigma_1, \dots,\sigma_{n-2}, 1)\\
&=\diag(\sigma, \sigma).
\end{align*}
To achieve the final form (\ref{eq:Schn}) of the inner product
we identify $ \RR ^{2(n-1)}\cong \CC ^{n-1}$
$$
\bar{z}^T = (\bar{x}_1, \dots , \bar{x}_{n-1}, \bar{y}_1, \dots , \bar{y}_{n-1})\cong (\bar{z}_1, \dots , \bar{z}_{n-1}) .
$$
The transformation
$$
(\bar{z}_1, \dots , \bar{z}_{n-1})\mapsto (e^{i\theta _1}\bar{z}_1, \dots , e^{i\theta _{n-1}}\bar{z}_{n-1})
$$
 is symplectic for any choice of angles $ (\theta _1, \dots , \theta _{n-1}).$
If we choose angles\\
 $ \theta = (\theta _1, \dots , \theta _{n-1})$ in a way that
$$
e^{i\theta _k}\bar{z}_k  = x_k >0, \, k=1, \dots n-1,
$$
by the above identification, the symplectic transformation becomes:
$$
\bar{z}^T = (\bar{x}_1, \dots , \bar{x}_{n-1}, \bar{y}_1, \dots , \bar{y}_{n-1})\mapsto \bar{x}^T=(x_1, \dots, x_{n-1}, 0, \dots, 0)=(x^T,0).
$$
Let us be more precise. In order to explicitly find the last transformation, we define diagonal matrices
$$
A (\theta)  = \diag (\cos \theta _1, \dots ,\cos \theta _{n-1}), \quad B (\theta )= \diag (\sin \theta _1, \dots ,\sin \theta _{n-1})
$$
and take $M(\theta )$ to be the block matrix
$$
M(\theta ) = \begin{pmatrix}
A(\theta ) & -B(\theta )   \\
B(\theta ) & A(\theta )
\end{pmatrix}.
$$
Then $M = M (\theta )$ is symplectic matrix, hence using $F_{Sp_5}(M)\in Sp$  we get
\begin{equation*}
\arraycolsep=4.5pt
S_5=F_{Sp_5}^TS_4F_{Sp_5}=
\begin{pmatrix}
p & \bar{z}^TM & 0  \\
M^T\bar{z} & M^T\bar{D} M & 0 \\
0 & 0 & \beta
\end{pmatrix} =
\begin{pmatrix}
p & \bar{x}^T & 0 \\
\bar{x} & \bar{D} & 0 \\
0 & 0  &\beta
\end{pmatrix}
=
\begin{pmatrix}
p &x^T& 0 & 0\\
x& \sigma & 0 & 0\\
0 & 0 &\sigma&0\\
0 & 0 & 0 & \beta
\end{pmatrix}
\end{equation*}
which is the final form from the statement of the theorem.
\end{proof}

\begin{Remark}
Every inner product on Lie algebra uniquely determines a left invariant metric on Lie group. Thus, the previous theorem gives us the classification of all Riemannian metric on the Lie group $\Cg$.
\end{Remark}

\section{Curvature properties of the Lie group $\Cg$ }
\label{sec:geometry}

We have classified all the non-isometric left invariant Riemannian metrics on $\Cg$ and now we will show some useful curvature properties of these metrics. Let $(X, Y_1,\dots ,Y_{n-1}, Z_1,\dots, Z_{n-1}, W)$ be a left invariant basic of $\chn$ with commutators of the form (\ref{com_chn}).
Denote by $g(p,x,\sigma,\beta)$ the left invariant metrics defined by the inner product $S(p,x,\sigma,\beta)$ in the given basis.

Let $\nabla$ be its Levi-Civita connection. For any left invariant vector fields $X_1$, $X_2$, $X_3$, the Koszul's formula reduces to
\begin{equation}
\label{Koszul}
2g (\nabla _{X_1} X_2, X_3)  = g ([X_1, X_2 ], X_3)- g ([X_2, X_3], X_1 )+ g ([X_3, X_1], X_2).
\end{equation}

From Koszul's formula (\ref{Koszul}) and the fact that $\nabla$ is torsion-free, i.e.   $\nabla _{X_1}X_2 - \nabla _{X_2}X_1 = [X_1, X_2]$, we find all nonzero covariant derivatives:

\begin{eqnarray*}
&&
\nabla_XX=\frac{1}{2z}\left(\sum_{i=1}^{n-1} \frac{x_i^2}{\sigma_i} X- p\sum_{i=1}^{n-1} \frac{x_i}{\sigma_i}Y_i\right), \quad
\nabla_XY_i=\frac{x_i}{2z}\left(X-\sum_{k=1}^{n-1} \frac{x_k}{\sigma_k}Y_k \right), \\
&&
\nabla_{Y_i}X=\frac{x_i}{2z}\left(X-\sum_{k=1}^{n-1} \frac{x_k}{\sigma_k}Y_k \right) - \frac{1}{2}Y_i, \quad
\nabla_{Y_i}Y_j=\frac{\delta_{ij} \sigma_i}{2z}\left(X-\sum_{k=1}^{n-1} \frac{x_k}{\sigma_k}Y_k \right), \\
&&
\nabla_{Y_i}Z_j=-\frac{\delta_{ij}}{2}W =-\nabla_{Z_j}Y_i, \quad
\nabla_{Y_i}W=\frac{\beta}{2\sigma_i}Z_i=\nabla_WY_i, \quad
\nabla_{Z_i}X=-\frac{1}{2}Z_i, \\
&&
\nabla_{Z_i}Z_j=\frac{\delta_{ij} \sigma_i}{2z}\left(X-\sum_{k=1}^{n-1} \frac{x_k}{\sigma_k}Y_k \right),\quad
\nabla_WW=\frac{\beta}{z}\left(X-\sum_{i=1}^{n-1} \frac{x_i}{\sigma_i}Y_i \right),\\
&&
\nabla_WX=-W, \quad
\nabla_WZ_i=\frac{\beta}{2z\sigma_i}\left(x_iX-x_i\sum_{k=1}^{n-1} \frac{x_k}{\sigma_k}Y_k-zY_i \right)=\nabla_{Z_i}W, \\
\end{eqnarray*}
where $z=p-\sum_{i=1}^{n-1} \frac{x_i^2}{\sigma_i}, \quad \sigma_{n-1}=1$.

The Riemann curvature operator, for all $X_1, X_2\in \chn$, is defined by
$$
R(X_1,X_2)=\nabla_{X_1}\nabla_{X_2}-\nabla_{X_2}\nabla_{X_1}-\nabla_{[X_1,X_2]}.
$$
The expressions for the Riemann curvature operators are very complex, so it is convenient to express them in terms of the wedge product (for explicit formulas see appendix).
Since we work with left invariant basis, we have a natural identification \mbox{$T_p\Cg \cong \chn$} for any $p\in \Cg .$F
Using the fact that $g(R(X_1,X_2)X_3,X_4) = -g(R(X_1,X_2)X_4,X_3)$, we know that Riemann curvature operator belongs to the algebra of skew-symmetric endomorphisms $so(g).$ Consequently, we identify the skew-symmetric endomorphisms $so(g)$ with 2-vectors $\Lambda ^2 T_p\Cg $ by the formula
$$
(X_1\wedge X_2)(X_3) := g(X_2, X_3)X_1 - g(X_1,X_3)X_2,
$$
for any $X_3\in T_p\Cg.$
Since $R(X_1,X_2) = -R(X_2,X_1)$, the curvature tensor can be regarded as skew-symmetric operator on the space of 2-vectors
$$
R:\Lambda ^2 T_p\Cg \to \Lambda ^2 T_p\Cg\cong so(g), \quad \quad R(X_1\wedge X_2) := R(X_1,X_2).
$$

\begin{Lemma}\label{klin}
The Riemann curvature operators \mbox{$R:  \Lambda ^2 \chn \to \Lambda ^2 \chn$}
on a basis of the space of 2-vectors $\Lambda ^2 \chn$ are:
\begin{eqnarray*}
R(X, Y_i)=&&-\frac{1}{4z}X\wedge Y_i-\frac{1}{4\sigma_i}Z_i \wedge W,
\\
R(X, Z_i)=&& \frac{1}{4z\sigma_i}
(-\sigma_iX\wedge Z_i-2x_i X\wedge W+x_i\sum_l\frac{x_l}{\sigma_l}Y_l\wedge W+zY_i\wedge W
),
\\
R(X, W)=&& \frac{1}{2z}[-2X\wedge W+\sum_l\frac{x_l}{\sigma_l}Y_l\wedge W\\
&&+\beta\sum_m\frac{1}{\sigma_m^2}(-\frac{3}{2}x_m X\wedge Z_m+x_m \sum_l\frac{x_l}{\sigma_l}Y_l\wedge Z_m+zY_m\wedge Z_m
)
 ],
\end{eqnarray*}
\begin{eqnarray*}
R(Y_i,Y_j)=&& -\frac{1}{4z}Y_i\wedge Y_j-\frac{\beta}{4\sigma_i\sigma_j}Z_i\wedge Z_j,
\\
R(Y_i, Z_j)=&&\frac{1}{4z \sigma_i \sigma_j}[
-x_j\sigma_iY_i\wedge W+2\delta_{ij}\sigma_i\sigma_j(X\wedge W-\sum_l\frac{x_l}{\sigma_l}Y_l\wedge W)-\sigma_i\sigma_jY_i\wedge Z_j\\
&&+2\delta_{ij}\sigma_i \sigma_j\beta\sum_m[\frac{1}{\sigma_m^2}(x_mX\wedge Z_m-x_m\sum_l\frac{x_l}{\sigma_l}Y_l\wedge Z_m-zY_m\wedge Z_m)]\\
&&+\beta(x_jX\wedge Z_i-x_j\sum_l\frac{x_l}{\sigma_l}Y_l\wedge Z_i-zY_j\wedge Z_i)
],
\\
R(Y_i, W)=&& \frac{\beta}{4z\sigma_i}[
\frac{x_i}{\sigma_i}(-X\wedge W+\sum_l\frac{x_l}{\sigma_l}Y_l\wedge W)+
(\frac{z}{\sigma_i}-\frac{2\sigma_i}{\beta})Y_i\wedge W\\
&&+X\wedge Z_i-\sum_l\frac{x_l}{\sigma_l}Y_l\wedge Z_i-\sigma_i\sum_l\frac{x_l}{\sigma_l^2}Y_i\wedge Z_l
],
\\
R(Z_i, Z_j)=&&\frac{1}{4z \sigma_i \sigma_j}[\left(x_i\sigma_j Z_j\wedge W-x_j\sigma_i Z_i\wedge W\right)
-\sigma_i\sigma_jZ_i\wedge Z_j\\
&&
+\beta (x_i X\wedge Y_j -x_jX\wedge Y_i-\sum_l \frac{x_l}{\sigma_l}(x_iY_l\wedge Y_j-x_jY_l\wedge Y_i)-zY_i \wedge Y_j )
],
\\
R(Z_i,W)=&&\frac{\beta}{4z}(
-\frac{1}{\sigma_i}X\wedge Y_i+\frac{1}{\sigma_i}\sum_l\frac{x_l}{\sigma_l}Y_l\wedge Y_i
-\sum_l\frac{x_l}{\sigma_l^2}Z_l\wedge Z_i\\
&&+(\frac{z}{\sigma_i^2}-\frac{2}{\beta})Z_i\wedge W
+\frac{x_i}{\sigma_i}\sum_l\frac{x_l}{\sigma_l^2}Z_l\wedge W
).
\end{eqnarray*}
\end{Lemma}

\begin{proof}
For example, we will prove formula for $R(Zi, W)$:
$$
R(Z_i,W)X=\frac{\beta}{4z\sigma_i}\left(-x_iX+x_i\sum_l\frac{x_l}{\sigma_l}Y_l+zY_i)\right).$$
Only wedge product applied on $X$ which give us vector fields $X$ and $Y_i$ are $(X\wedge Y_i )X=x_iX-pY_i$ and $(Y_i\wedge Y_j )X=x_jY_i-x_iY_j$. Because of that $A(X)=\mu (X\wedge Y_i)X+\eta(Y_l\wedge Y_i)X$ must be the same with regular expression of Riemann curvature operator. So, we have the following equation to solve:
\begin{equation}
\label{koefbeta}
\frac{\beta}{4z\sigma_i}\left(-x_iX+x_i\sum_l\frac{x_l}{\sigma_l}Y_l+zY_i
\right)=\mu (x_iX-pY_i)+\eta(x_lY_i-x_iY_l).
\end{equation}
Now, the part with vector fields  $X$ gives us coefficient $\mu$:
$$\mu=-\frac{\beta}{4z\sigma_i}.$$
After replacing the coefficient $\mu$ and $p=\sum_l\frac{x_l^2}{\sigma_l}+z$, in equation (\ref{koefbeta}), we get the same part with vector field $Y_i$:
$$\frac{\beta}{4z\sigma_i}\left(x_i\sum_l\frac{x_l}{\sigma_l}Y_l+zY_i\right)=-\frac{\beta}{4z\sigma_i} (-\sum_l\frac{x_l^2}{\sigma_l}+z)Y_i+\eta(x_lY_i-x_iY_l).$$
Undoubtedly, that coefficient is $\eta=\frac{\beta}{4z\sigma_i}\sum_l\frac{x_l}{\sigma_l}$.  Therefore, after verification we get :
$$A=-\frac{\beta}{4z\sigma_i} (X\wedge Y_i)+\frac{\beta}{4z\sigma_i}\sum_l\frac{x_l}{\sigma_l}(Y_l\wedge Y_i).$$
It is necessary to check what we have from $A(Y_i), A(Z_i),  A(W)$ in Riemann curvature operators  $R(Z_i,W)Y_j, R(Z_i,W)Z_j, R(Z_i,W)W$ and by the analogous procedure, get the rest of the expression in terms of the wedge product. All the rest curvature operators are found by the similar lengthy calculations.
\end{proof}

Using the definition of the Ricci curvature tensor
$$
\label{eq:ro}
Ric(X_1,X_2)=Tr(X \mapsto R(X,X_1)X_2), \quad \forall X_1, X_2 \in \chn$$
and the scalar curvature
$$\tau=\sum g^{ij}r_{ij},$$ we calculate them directly.

\begin{Lemma}
For all left-invariant Riemannian metrics $g(p,x,\sigma,\beta)$
\begin{itemize}
\item[i)]
the Ricci curvature tensor in a basis with commutators (\ref{com_chn}) is

\footnotesize
\arraycolsep=0pt
\begin{equation}
\label{Ricci}
Ric=
{-}\frac{1}{2z}
\begin{pmatrix}
np{+}z & nx^T & 0 & 0\\
nx & n\sigma{+}\beta z \sigma^{-1} & 0 & 0\\
0 & 0 & n\sigma {+}\beta z \sigma^{-1}{+}\beta v v^T & (2n+1)\frac{\beta}{2} v\\
0 & 0 & (2n+1)\frac{\beta}{2} v^T  &
2n\beta{-}\beta^2\sum_{k=1}^{n-1}\frac{1}{\sigma_k^2}\left(\frac{x_k^2}{\sigma_k}{+}z\right)
\end{pmatrix},
\end{equation}
\normalsize
where\\
$x=(x_1, \dots, x_{n-1})^T, \quad
\sigma=\diag(\sigma_1, \dots, \sigma_{n-2},1), \quad \sigma_{n-1}=1$, \\
$v=\sigma^{-1} x=\left( \frac{x_i}{\sigma_i} \right), \quad
z=p-\sum_{i=1}^{n-1}\frac{x_i^2}{\sigma_i}.$

\item[ii)]
$\Cg$ has a constant negative scalar curvature:
$$\tau=-\frac{1}{2z}\left[2n^2+n+1+\beta\sum_{i=1}^{n-1} \frac{1}{\sigma_i^2} \left( z+\frac{x_i^2}{\sigma_i} \right) \right].$$
\end{itemize}
\end{Lemma}
\begin{proof}
This follows directly using calculation for Levi-Civita connection above and Riemann curvature operators from Lemma \ref{klin}. \\
Since $\sigma_i\geq1, \enskip \forall i\in\{1,\dots, n-1\}$, and $\beta, z>0$,
scalar curvature is strictly negative.
\end{proof}

\begin{Remark}
This is consistent with the Milnor's results on the scalar curvature of a non-flat left invariant Riemannian metric on solvable Lie groups (see \cite{Milnor}, Theorem 3.1).
\end{Remark}

\begin{Remark}
Ricci negative metrics on $\Cg$ exist. For example, if $x_i=0$ and $2n-p\beta\sum_{i=1}^{n-1}\frac{1}{\sigma_i^2}>0$, all eigenvalues are negative.\\
$$
Ric=-\frac{1}{2p}\diag\left((n+1)p, u , u, 2n\beta-\beta^2 p \sum_{k=1}^{n-1}\frac{1}{\sigma_k^2}\right),
$$
$u=\left( n \sigma_1+\frac{\beta p}{\sigma_1},\dots, n \sigma_{n-2}+\frac{\beta p}{\sigma_{n-2}}, n+\beta p  \right) \in \RR^{n-1} .$
\end{Remark}
\begin{Definition}
Metric is Einstein if the Ricci tensor is proportional to the metric tensor.
\end{Definition}
\begin{Theorem}
\label{einstein}
The left-invariant Riemannian metric $g(p,x,\sigma,\beta)$ on $\Cg$ is Einstein if and only if the following conditions are satisfied:
$$ p\beta=1, \quad (\forall  i) \;  x_i=0, \quad \sigma_i=1,$$
i.e.\ the corresponding inner product is
$$S=\diag\left( p, 1, \dots, 1, \frac{1}{p}\right).$$
\end{Theorem}

\begin{Remark}
The Einstein metric on solvable Lie group $\Cg$ is unique up to the action of the automorphism group and scaling, which is consistent with the Lauret's results on solsolitons \cite{Lauretsolsoliton}, since the Einstein metric is the trivial example of solsoliton. Of course, this is the standard K\"ahler metric of the complex hyperbolic space $\Chn$.
\end{Remark}

\subsection{Extension of Ricci solitons on Heinsenberg group to Einstein metric on $\Cg$}

Lauret showed the general way how to obtain solsolitons by abelian extension of nilsolitons.
Since nilsolitons on Heisenberg are classified (both Riemann and Lorentz case, see \cite{Nasehi}),
and $\Cg$ is one-dimensional  abelian extension of Heisenberg group,
it is interesting as an example to apply the  Lauret's construction \cite{Lauretsolsoliton} to see how the extension of nilsoliton of Heisenberg gives rise to Einstein metric on $\Cg$.

\begin{Definition}[\cite{Lauretsolsoliton}]
A left-invariant metric $g$ on a simply connected solvable (nilpotent) Lie group is called a solsoliton (nilsoliton) if the corresponding Ricci operator satisfies $$Ric(g)=cI + D,$$ for some $c\in \RR$, $D \in Der(\mathfrak{g})$ and $Ric(g)$ denotes the Ricci operator of $g$.
\end{Definition}

\begin{customthm}{4.3}[\cite{Lauretsolsoliton}]
Let $(\nn,\langle\cdot,\cdot\rangle_1)$ be a nilsoliton, say with Ricci operator $\Ric_1=cI+D_1$,  $c<0$, $D_1\in \Der(\nn)$ and consider $\aa$ any abelian Lie algebra of symmetric derivations of $(\nn,\langle\cdot,\cdot\rangle_1)$. Then the solvmanifold $S$ with Lie algebra $\ss=\aa\oplus \nn$ (semidirect product) and inner product given by
$$\langle\cdot,\cdot\rangle|_{n \times n}=\langle\cdot,\cdot\rangle_1, \qquad \langle \aa, \nn\rangle=0,
\qquad \langle A, A \rangle= -\frac{1}{c}tr A^2, \qquad \forall A \in \aa,$$
is a solsoliton with $\Ric=cI+D$, where $D\in \Der(\ss)$ is defined by $D|_\aa=0$, $D|_\nn=D_1-\ad H|_n$ and $H$ is the mean curvature vector of $S$. Futhermore, $S$ is Einstein if and only if $D_1\in \aa$.
\end{customthm}

In order to apply the previous Proposition to $\chn =\RR \ltimes \hh$, we consider it as an orthogonal sum  $\aa\oplus \nn$. Thus, we switch from the standard base\\
$(X,Y_1,\dots,Y_{n-1}, Z_1,\dots,Z_{n-1},W)$ to the orthonormal one\\
 $(e,f_1,\dots,f_{n-1}, g_1,\dots,g_{n-1},w)$,
$$
e=\frac{V}{\Vert V\Vert}, \quad f_{i}=\frac{Y_i}{\sqrt{\sigma_i}}, \quad g_{i}=\frac{Z_j}{\sqrt{\sigma_j}}, \quad w=\frac{W}{\sqrt{\beta}},
$$
where
\begin{eqnarray*}
V&=& X-\sum g(X,Y_i)\frac{Y_i}{\Vert Y_i\Vert}-\sum g(X,Z_i)\frac{Z_i}{\Vert Z_i\Vert}- g(X,W)\frac{W}{\Vert W\Vert}\\
  &=& X- \sum \frac{x_i}{\sqrt{\sigma_i}}Y_i,\\
  g(V,V)&=&p-2 \sum\frac{x_i^2}{\sqrt{\sigma_i}}+\sum{x_i^2}
\end{eqnarray*}

In our case nilsoliton is Heisenberg $\nn=\hh$ and the abelian part is one dimensional spanned by the vector $e$. From the Theorem 3.4. in \cite{Nasehi}, applied to dimension $2n-1$,
metric $\langle\cdot,\cdot\rangle_1$ on Heisenberg is given by $\diag(1,\dots, 1, \beta)$, i.e.\ all $\sigma_i=1$,
$D_1=n \beta \diag(\frac{1}{2},\dots \frac{1}{2}, 1)$,
$\Ric_1=-\frac{\beta}{2}\diag(\frac{1}{2},\dots, \frac{1}{2}, 1-n), c=-\frac{\beta}{2}({n+1})$,
$V= X- \sum x_iY_i$,
$g(V,V)=p-\sum{x_i^2}=z$. Mean curvature vector of solvmanifold $S$ is the only element $H \in \aa$ such that $\langle H, A\rangle =\tr \, \ad A$ for any $A \in \aa$. Direct calculation gives us that
 $H=\frac{n}{z}e$, \\

$$\ad H=\frac{n}{z}
\begin{pmatrix}
0 & 0 & 0 & 0\\
0 & d_\frac{1}{2} & 0 & 0\\
0 & 0 & d_\frac{1}{2} & 0\\
0 & 0 & x\sqrt{\beta} & 1
\end{pmatrix},
$$
where
$d_\frac{1}{2}=\diag(\frac{1}{2}, \dots, \frac{1}{2})$, $x=(x_1, \dots, x_{n-1})$ and

$$
D|_\nn=D_1-\ad H|_n=\frac{n}{z}
\begin{pmatrix}
\frac{\beta z -1}{2} & 0 & 0\\
0 & \frac{\beta z - 1}{2} & 0\\
0 & -x\sqrt{\beta} & \beta z - 1
\end{pmatrix}.
$$
If we identify
$\aa=\RR\langle e\rangle \cong \RR\langle \ad e\rangle$, and calculate
$$\ad e=\frac{1}{\sqrt{z}}
\begin{pmatrix}
0 & 0 & 0 & 0\\
0 & d_\frac{1}{2} & 0 & 0\\
0 & 0 & d_\frac{1}{2} & 0\\
0 & 0 & x\sqrt{\beta} & 1
\end{pmatrix},
\enskip
(\ad e)^2=\frac{1}{z}
\begin{pmatrix}
0 & 0 & 0 & 0\\
0 & d_\frac{1}{4} & 0 & 0\\
0 & 0 & d_\frac{1}{4} & 0\\
0 & 0 & \frac {3}{2}x\sqrt{\beta} & 1
\end{pmatrix},
$$
$$
\tr\,\ad e^2=\frac{n+1}{2z},$$
then, from the condition $\langle A, A \rangle= -\frac{1}{c}tr A^2, \enskip \forall A \in \aa,$ we have
$$1=\langle e, e \rangle= -\frac{1}{c}\tr\,\ad e^2=\beta z.$$
Using the condition $\Ric=cI+D$, already calculated Ricci tensor (\ref{Ricci}) on $\chn$, and the results above, now
we get that all $x_i=0$, $D=0$, $z=p$, so $p\beta=1$.
Therefore, there is only one Ricci soliton on $\chn$  up to homothety and action of automorphism group. That is exactly the Einstein metric from the classification $\diag(\frac{1}{\beta},1,\dots,\beta)$, which is an alternative proof of the Theorem \ref{einstein}.

\begin{Remark}
Since we know that Einstein metric from the classification is the simplest Ricci soliton,
this result is consistent with the Corollary 4.10 and Theorem 5.1 from \cite{Lauretsolsoliton}, saying that up to isometry, all solsolitons can be constructed as in Proposition 4.3.
\end{Remark}

\textbf{Acknowledgments}:
This research is partially supported
by Ministry of Education, Science and Technological Development of Republic
of Serbia,
within the project 174012,
through Faculty of Mathematics, University of Belgrade and
Mathematical Institute of the Serbian Academy of Sciences and Arts.

\appendix
\section{Appendix: Explicit expression of Riemann curvature operator}

\begin{eqnarray*}
R(X,Y_i)X&&=\frac{1}{4z}( -x_iX+pY_i) \\
R(X,Y_i)Y_j&&=\frac{1}{4z}(x_jY_i-\delta_{ij}\sigma_iX)\\
R(X,Y_i)Z_j&&=\frac{\delta_{ij}}{4}W\\
R(X,Y_i)W&&=-\frac{\beta}{4 \sigma_i}Z_i\\
\\
R(X, Z_i)X&&=\frac{p}{4z}\left(Z_i+\frac{x_i}{\sigma_i}W\right)\\
R(X, Z_i)Y_j&&=\frac{x_i}{4z}\left(Z_i+\frac{x_i}{\sigma_i}W-\delta_{ij}\frac{z}{x_j}W\right)\\
R(X, Z_i)Z_j&&=-\frac{\sigma_i}{4z}\delta_{ij}X\\
R(X, Z_i)W&&=\frac{\beta}{4z\sigma_i}\left(-2x_iX+x_i\sum_l\frac{x_l}{\sigma_l}Y_l+zY_i\right)\\
\\
R(X,W)X&&=\frac{p\beta}{4z}\sum_l\frac{x_l}{\sigma_l^2}Z_l+ \frac{z+p}{2z}W\\
R(X,W)Y_i&&=\frac{x_i\beta}{4z}\sum_l\frac{x_l}{\sigma_l^2}Z_l-\frac{\beta}{2\sigma_i}Z_i+\frac{x_i}{2z}W\\
R(X,W)Z_i&&=\frac{\beta}{4z\sigma_i}\left(-3x_iX+2x_i\sum_l\frac{x_l}{\sigma_l}Y_l+2zY_i
\right)\\
R(X,W)W&&=\frac{\beta}{z}\left(-X+\frac{1}{2}\sum_l\frac{x_l}{\sigma_l}Y_l\right)\\
\\
R(Y_i,Y_j)X&&=\frac{1}{4z}(x_iY_j-x_jY_i)\\
R(Y_i,Y_j)Y_k&&=\frac{1}{4z}(\delta_{ik}\sigma_i Y_j-\delta_{jk}\sigma_j Y_i)\\
R(Y_i,Y_j)Z_k&&=\frac{\beta}{4}\left(\frac{\delta_{ik}}{\sigma_j}Z_j-\frac{\delta_{jk}}{\sigma_i}Z_i\right)\\
R(Y_i,Y_j)W&&=0\\
\\
\end{eqnarray*}

\begin{eqnarray*}
R(Y_i,Z_j)X&&=\frac{1}{4z}\left[x_iZ_j+\left(\frac{x_ix_j}{\sigma_j}-2\delta_{ij}z\right)W \right]\\
R(Y_i,Z_j)Y_k&&=\frac{1}{4}
\left( \delta_{jk}\frac{\beta}{\sigma_i}Z_i
+\delta_{ik}\frac{ \sigma_i}{z}Z_j
+2\,\delta_{ij}\frac{\beta}{\sigma_k}Z_k
+\delta_{ik}\frac{x_j \sigma_i}{z\sigma_j}W
\right)\\
R(Y_i,Z_j)Z_k&&=\frac{1}{4z}\left[
2\,\delta_{ij}\frac{\beta}{\sigma_k}\left(x_kX-zY_k-x_k\sum_l\frac{x_l}{\sigma_l}Y_l\right)
+\delta_{ik}\frac{\beta}{\sigma_j}\left(x_jX-zY_j-x_j\sum_l\frac{x_l}{\sigma_l}Y_l\right)
-\delta_{jk}\sigma_jY_i
\right]\\
R(Y_i,Z_j)W&&=\frac{\beta}{4z}\left[-\frac{x_j}{\sigma_j}Y_i+2\,\delta_{ij}(X-\sum_l\frac{x_l}{\sigma_l}Y_l)
\right]\\
R(Y_i,W)X&&=\frac{1}{4z}\left(\beta x_i\sum_l\frac{x_l}{\sigma_l^2}Z_l-\frac{\beta z}{\sigma_i}Z_i+2\,x_iW
\right)
\\
R(Y_i,W)Y_j&&=\frac{\beta}{4z}\delta_{ij}\left[\sigma_i\sum_l\frac{x_l}{\sigma_l^2}Z_l+(2\,\frac{\sigma_i}{\beta}-\frac{z}{\sigma_i})W
\right]\\
R(Y_i,W)Z_j&&=\frac{\beta}{4z}\left[\delta_{ij}(X-\sum_l\frac{x_l}{\sigma_l}Y_l)-\frac{x_j}{\sigma_j}Y_i
\right]\\
R(Y_i,W)W&&=\frac{\beta^2}{4z\sigma_i^2}\left[-x_iX+x_i\sum_l\frac{x_l}{\sigma_l}Y_l+(z-2\frac{\sigma_i^2}{\beta})Y_i
\right]\\
\\
R(Z_i,Z_j)X&&=0
\\
R(Z_i,Z_j)Y_k&&=\frac{\beta}{4z}\left[
\frac{\delta_{ik}}{\sigma_j}(-x_jX+x_j\sum_l\frac{x_l}{\sigma_l}Y_l+zY_j)
-\frac{\delta_{jk}}{\sigma_i}(-x_iX+x_i\sum_l\frac{x_l}{\sigma_l}Y_l+zY_i)
\right]
\\
R(Z_i,Z_j)Z_k&&=\frac{1}{4z}\left[
\delta_{ik}\sigma_i(Z_j+\frac{x_j}{\sigma_j}W)
-\delta_{jk}\sigma_j(Z_i+\frac{x_i}{\sigma_i}W)
\right]\\
R(Z_i,Z_j)W&&=\frac{\beta}{4z}\left(\frac{x_i}{\sigma_i}Z_j-\frac{x_j}{\sigma_j}Z_i
\right)\\
\\
R(Z_i,W)X&&=\frac{\beta}{4z\sigma_i}\left(-x_iX+x_i\sum_l\frac{x_l}{\sigma_l}Y_l+zY_i)
\right)\\
R(Z_i,W)Y_j&&=\frac{\beta}{4z}\delta_{ij}\left(-X+\sum_l\frac{x_l}{\sigma_l}Y_l
\right)\\
R(Z_i,W)Z_j&&=\frac{\beta}{4z}\left[
-\frac{x_j}{\sigma_j}Z_i+\delta_{ij}\sigma_i\sum_l\frac{x_l}{\sigma_l}^2Z_l
+\left(\delta_{ij}\left(2\frac{\sigma_i}{\beta}-\frac{z}{\sigma_j}\right )-\frac{x_ix_j}{\sigma_i\sigma_j}
\right)W
\right]\\
R(Z_i,W)W&&=\frac{\beta^2}{4z}\left[\left (\frac{z}{\sigma_i^2}-\frac{2}{\beta}\right)Z_i
+ \frac{x_i}{\sigma_i}\sum_l\frac{x_l}{\sigma_l^2}Z_l
\right]\\
\end{eqnarray*}


\begin{thebibliography}{20}

\bibitem{Ale1}{D. V. Alekseevsky, \textit{Conjugacy of polar factorizations of Lie groups}, Mat. Sbornik 84,  Volume 13, 1971, Pages 14-26.
}

\bibitem{Ale2}{D. V. Alekseevsky, \textit{Homogeneous Riemannian spaces of negative curvature}, Mat.  Sbornik 96,  Volume 25, 1975, Pages 87-109.
}

\bibitem{BB}{
L. B\'erard-B\'ergery, \textit{Homogeneous Riemannian Spaces of Dimension Four}, Seminar A. Besse,
Four-dimensional Riemannian geometry, 1985.
}

\bibitem{neda}{
N. Bokan, T. \v Sukilovi\'c, S, Vukmirovi\'c, \textit{Lorentz geometry of 4-dimensional nilpotent Lie
groups}, Geom. Dedicata. , Volume 177, % Issue 1,
2015, Pages 83-102.
}

\bibitem{calvaruso}{
G. Calvaruso, A. Zaeim, \textit {Four-dimensional Lorentzian Lie groups}, Differ. Geom. Appl., Volume 31, 
%Issue 4
2013, Pages 496-509.
}

\bibitem{ChPapa}{
T. Christodoulakis, G. O. Papadopoulos, A. Dimakis, \textit{Automorphisms of real four-dimensional
Lie algebras and the invariant characterization of homogeneous 4-spaces}, J. Phys. A, Math.
Gen., Volume 36, % Issue 2,
2002, Pages 427-441.
}

\bibitem {CorParker} {
L. A. Cordero, P. E. Parker, \textit{Left-invariant Lorentz metrics on 3-dimensional Lie groups},
Rend. Mat. Appl., Volume 17, 1997, Pages 129-155.
}

\bibitem{Vazquez}
{M. Dom\'inguez-V\'azquez, V. Sanmart\'in-L\'opez, H. Tamaru,
\textit{Codimension one Ricci soliton subgroups of solvable Iwasawa groups}, Journal de Math\' ematiques Pures et Appliqu\' ees, 2021
}

\bibitem {TamaruMilnorType} {
T. Hashinaga, H. Tamaru,  K. Terada, \textit{Milnor-type theorems for left-invariant Riemannian metrics
on Lie groups}, J. Math. Soc. Japan, Volume 68, %No. 2 pp. ,
2016, Pages 669-684.
}

\bibitem{Heintze}{
E. Heintze, \textit{On Homogeneous Manifolds of Negative Curvature}, Math. \mbox{Annalen}, Volume 211, 1974, Pages 23-34.
}

\bibitem{Heintze2}{
E. Heintze, \textit{Riemannsche Solvmannigfaltigkeiten}, Geom. Dedicata, Volume 1, 1973, Pages 141-147.
}

\bibitem {TamaruOrbite} {
H. Kodama, A. Takahara, H. Tamaru, \textit{The space of left-invariant metrics on a Lie group
up to isometry and scaling}, Manuscripta mathematica,Volume 135, 2011, Pages 229-243.
}

\bibitem{TamaruHxR}
{ Y. Kondo, H. Tamaru,
\textit{A classification of left-invariant Lorentzian metrics on some nilpotent Lie groups}, arXiv:2011.09118 [math.DG], 2020.
}

\bibitem {TamaruPseudoR} {
A. Kubo, K. Onda, Y. Taketomi, H. Tamaru, \textit{On the moduli spaces of left-invariant pseudo-Riemannian metrics on Lie groups}, Hiroshima Math. J., Volume 46, 2016, % No. 3,
Pages 357-374.
}

\bibitem{LauretOne}
{J. Lauret, \textit{Degenerations of Lie algebras and geometry of Lie groups}, Differential Geom. Appl., Volume 18 , %no. 2,
2003, Pages 177-194.
}

\bibitem{Lauretdim34}{
J. Lauret, \textit{Homogeneous nilmanifolds of dimension 3 and 4}, Geom. Dedicata, Volume 68, 1997, Pages 145-155.
}

\bibitem {Lauretnilsoliton}
{J. Lauret \textit{Ricci soliton homogeneous nilmanifolds}, Math. Ann., Volume 319, 2001, Pages 715-733.
}

\bibitem {Lauretsolsoliton} {
J. Lauret \textit{Ricci soliton solvmanifolds}, Journal f\"{u}r die reine und angewandte Mathematik, Volume 2011, 2011,  %issue 650
Pages 1-21.
}

\bibitem{Milnor}{
J. Milnor, \textit{Curvatures of left invariant metrics on Lie groups}, Advances in Mathematics, Volume 21, % Issue 3,
1976, Pages 293-329.
}

\bibitem{Nasehi}{
M. Nasehi, \textit{On the Geometry of Higher Dimensional Heisenberg Groups}, Mediterranean Journal of Mathematics, Volume 16, 2019, Page 29.
}

\bibitem{Rahmani}{
S. Rahmani, \textit{Metriques de Lorentz sur les groupes de Lie unimodulaires de dimension 3}, J. Geom. Phys., Volume 9, 1992, Pages 295-302.
}

\bibitem {Tijana4sol} {
T. \v Sukilovi\' c, \textit{Classification  of left invariant metrics on 4-dimensional solvable Lie groups}, Theoretical and applied mechanics,
Volume 47, 2020, % Issue 2,
Pages 181-204.
}

\bibitem {Tijana4nil} {
T. \v Sukilovi\' c, \textit{Geometric propertis of neutral signature metrics on  4-dimensional nilpotent
Lie groups}, Revista de la Uni\'on matem\'atica Argentina, Volume 57, %No. 1,
2016, Pages 23-47.
}

\bibitem{Srdjan}{
S. Vukmirovi\'c, \textit{Classification of left-invariant metrics on the Heisenberg group}, Journal of Geometry and Physics, Volume 94, 2015, Pages 72-80.
}

\bibitem{TijanaiSrdjan}{
S. Vukmirovi\'c, T. \v Sukilovi\'c, \textit{Geodesic completeness of the left-invariant metrics on $\Rhn$},
Ukrainian Mathematical Journal, Volume 72, %No. 5, October,
2020, %(Ukrainian Original Vol. 72, No. 5, May, 2020)
Pages 702-711.
}

\bibitem{Williamson}
{J. Williamson, \textit{On the Algebraic Problem Concerning the Normal Forms of Linear Dynamical Systems},  American Journal of Mathematics,  Volume 58,  %Issue 1,
1936, Pages 141-163.
}

\end{thebibliography}
\end{document}